\documentclass[hidelinks,12pt,reqno,oneside]{amsart}
\usepackage{amsfonts,amssymb,latexsym,xspace,epsfig,graphics,color,soul}
\usepackage{amsmath,enumerate,stmaryrd,xy}
\usepackage{bbm}
\usepackage{natbib}
\usepackage{subfigure}
\usepackage{hyperref}
\usepackage{epigraph}
\DeclareFontFamily{U}{matha}{\hyphenchar\font45}
\DeclareFontShape{U}{matha}{m}{n}{
      <5> <6> <7> <8> <9> <10> gen * matha
      <10.95> matha10 <12> <14.4> <17.28> <20.74> <24.88> matha12
      }{}
\DeclareSymbolFont{matha}{U}{matha}{m}{n}

\DeclareMathSymbol{\Lt}{3}{matha}{"CE}
\DeclareMathSymbol{\Gt}{3}{matha}{"CF}

\newcommand{\X}{\mathcal{X}}

\newcommand{\N}{\mathbb{N}}

\newcommand{\F}{\mathcal{F}}
\newcommand{\B}{\mathcal{B}}
\newcommand{\Lamb}{\Lambda}

\newcommand{\R}{\ensuremath{\mathbb{R}}}     
\newcommand{\Z}{\ensuremath{\mathbb{Z}}}

\newtheorem{remark}{Remark}
\newtheorem{lemma}{Lemma}

\newtheorem{defi}{Definition}
\newtheorem{theo}{Theorem}
\newtheorem*{theo*}{Theorem}
\newtheorem{prop}{Proposition}
\newtheorem{coro}{Corollary}

\begin{document}
\DeclareGraphicsExtensions{.pdf,.gif,.jpg}



\title[Uniqueness of compatible stationary probability]{Uniqueness of stationary compatible probability measures for chains of infinite order with forbidden transitions}
\author{Gallesco, C.$^1$}
\address{$^1$Departmento de Estat\'istica, Instituto de Matem\'atica, Estat\'istica e Computa\c{c}\~ao Cient\'ifica, Universidade Estadual de Campinas, Brazil}
\email{gallesco@unicamp.br}
 
 \author{Gallo, S.$^2$}
\address{$^2$Departamento de Estat\'istica, Universidade Federal de S\~ao Carlos, Brazil}
\email{sandro.gallo@ufscar.br}

 \author{Takahashi, D. Y.$^3$}
\address{$^3$Instituto do C\'erebro, Universidade Federal do Rio Grande do Norte, Brazil}
\email{takahashiyd@gmail.com}



\maketitle

\epigraph{\emph{In honor of Antonio Galves who introduced us to this beautiful topic}}

\begin{abstract}
In this paper, we consider chains of infinite order on countable state spaces with prohibited transitions. We give a set of sufficient conditions on the structure of the probability kernels of the chains to have at most one stationary  probability measure compatible with the kernel. Our main result extends the uniqueness $\ell^2$ criterion from \cite{johansson/oberg/2003} which was obtained for strongly non-null chains. A particular attention is given to concrete examples, illustrating the main theorem and its corollaries, with comparison to results of the existing literature.
\end{abstract}


\section{Introduction}

Markov chains are stochastic processes characterized by transition probability kernels that describe how a finite portion of the past conditionally defines the probability of a present event. A natural generalization, that we call here chains of infinite order (see \cite{fernandez/maillard/2005} for alternative terminology), allows for transition probabilities that depend on the entire infinite past. These  non-Markovian processes can exhibit complex behaviors not present in finite-order Markov chains, such as non-existence and non-uniqueness (phase transition) of compatible stationary measures, even when the state space, also called {\it alphabet} in this context, is finite.


Much of the existing literature focuses on chains of infinite order with (1) continuous transition probabilities —that is, having some decay of dependence on the past— and (2) some non-nullness assumption like strong non-nullness, where all transitions are permitted, and weak non-nullness, where there exists at least one symbol to which transitions are always permitted. 
General conditions for the existence and uniqueness of compatible stationary measures in non-null continuous cases are well understood \cite{keane/1972,ledrappier/1974,walters/1975, berbee/1986,bressaud_fernandez_galves_1999a, bressaud_fernandez_galves_1999b,comets_fernandez_ferrari_2002,johansson/oberg/2003}. However, for kernels which are null and/or discontinuous, the theory is significantly less developed. There are few works that addressed the discontinuous kernel case (see \cite{ferreira2020non} and references therein). Nevertheless, for null kernels, even under the continuity assumption, we find only partial results and specific examples in the literature. For instance, \cite{fernandez/maillard/2005} obtain uniqueness under a Dobrushin type assumption for compact alphabet. \cite{johansson/oberg/pollicot/2007} give a general condition ensuring uniqueness of compatible measure, but which is not readily verifiable from the structure of the associated transition probability kernel. Finally \cite{desantis/piccioni/2010}  give conditions ensuring perfect simulation of the unique compatible stationary measure, which is  a more restrictive condition than  uniqueness.

Thus, a systematic study of null kernels for chains of infinite order is still missing. The aim of this article is to start filling this gap by providing verifiable uniqueness conditions when there are prohibited transitions. Our main theorem holds for kernels on countable alphabets satisfying an almost sure square summability condition, and a set of conditions on the $0$'s of the kernel that generalize the conditions of ergodic Markov chains. Under these assumptions, our theorem states that there exists at most one stationary compatible measure. Moreover, we prove that the compatible measure is $\beta$-mixing (equivalently, absolutely regular or weak Bernoulli). It is important to note that our square summability assumption on the kernel is reminiscent of \cite{johansson/oberg/2003,johansson/oberg/pollicot/2007,GGT2018}, and allows discontinuities (see examples in Section \ref{sec:examples}). Our condition is more general than the condition in \cite{desantis/piccioni/2010}, and, when the alphabet is finite, it is complementary to the Dobrushin-type condition obtained in \cite{fernandez/maillard/2005} (see \cite{GGT2018} for a more detailed discussion).

To prove our result, we use the absolute continuity criterion for stochastic processes developed by Jacod and Shiryaev \citep{jacod/shiryaev/2002}. Their result allows us to conclude the absolute continuity of the entire processes from the local absolute continuity. The  absolute continuity of the entire process in turn implies uniqueness using the Ergodic Decomposition Theorem. When the transition kernel is strongly non-null, the local absolute continuity becomes trivial and we only need to check the almost sure square summability of the variation of the kernel (see \cite{johansson/oberg/pollicot/2007, GGT2018}). To the best of our knowledge, \cite{johansson/oberg/pollicot/2007} were the first to use the absolute continuity criterion developed by Jacod and Shiryaev to prove uniqueness of chains of infinite order. The key difference between our results  and the main result in \cite{johansson/oberg/pollicot/2007} is that they do not provide general verifiable criteria for local absolute continuity of the processes in the case of null kernels. Our work is motivated by their work and is complementary to it, since we give a proper set of conditions for null kernels ensuring that we have local absolute continuity. For this, we first characterize the structure of null transition probabilities, splitting between those prohibited transitions which depend on finite sequences and those which depend on the entire sequence of the past. Then, the main part of the proof consists in showing that if the kernel locally behaves like an irreducible and aperiodic Markov kernel (see Definition \ref{Deferi} below) and the isolated zeros of the kernels at infinity have limited influence on the future, then every two processes starting from different pasts are locally comparable.






Let us describe the structure of the paper. In Section \ref{sec:notation} we introduce the main notation and definitions. In Sections \ref{sec:results} and \ref{sec:examples} we state the main results and examples of applications, respectively. Finally in Section \ref{sec:proofs} we give the proofs of the results of Section \ref{sec:results}, and provide in the appendix some auxiliary results needed in our proofs.




\section{Notation and basic definitions} \label{sec:notation}
\noindent{\bf Basic notation.} Let $A$ be a countable set endowed with its discrete topology, $\mathcal{X}=A^{\Z}$, $\mathcal{X}_{-} = A^{\Z_-}$ where $\Z_-=\{0,-1,-2,\dots\}$. We endow $\mathcal{X}$ and $\mathcal{X}_-$ with their respective product topologies and their corresponding Borel $\sigma$-algebras $\F$ and $\F_-$. We also denote by $x_i$ the $i$-th coordinate of $x \in \mathcal{X}$ and for $-\infty< i \leq j<\infty$ we write $x^{j}_{i}:=(x_{i},\ldots, x_{j})$ and $x^{i}_{-\infty}:=(\dots,x_{i-1},x_{i})$. For $x\in \mathcal{X}$ and $y \in \mathcal{X}_{-}$, a \emph{concatenation} $yx^{j}_{i}$ is a new sequence $z\in \mathcal{X}_{-}$ with $z^{0}_{i-j} = x^{i}_{j}$ and $z^{i-j-1}_{-\infty} = y$. Let $\o$ be the neutral element of the concatenation operation, that is $ x\o= \o x=x$ for all $x\in \mathcal{X}_{-}$. We will use the convention that if $j<i$, $x_i^j=\o$. Note that we are using the convention that when we scan an element $x\in\mathcal{X}$ from the right to the left we go further into the past. For any $i\ge0$, the set $A^i$ denotes the set of strings of size $i$ ($A^0=\{\o\}$), with no reference concerning indexes, and $A^*:=\cup_{i\ge0}A^i$ denotes the set of finite strings.  Finally, for all $\Lambda \subset \Z$, we will denote by $\F_{\Lambda}$ the $\sigma$-algebra generated by the canonical projections $\eta_j:\X\to A$, defined by $\eta_j(x)=x_j$, for $j\in \Lambda$ and we denote by $\sigma:\mathcal{X}\rightarrow \mathcal{X}$ the shift map over $\mathcal{X}$, defined by $[\sigma(x)]_i=x_{i+1}$. \\

\noindent{\bf Probability kernels.} Consider a measurable function $g: \mathcal{X}_{-}\to \R$.  The \emph{variation rate} (or \emph{continuity rate}) of order $k\geq 1$ of $g$ is defined by
\begin{equation*}
\textrm{var}_{k}(g):=\sup_{x,y\in \mathcal{X}_-}\sum_{a\in  \mathcal{X}_-}|g(xa_{-k}^{0})-g(ya_{-k}^{0})|. 
\end{equation*}
In the case $A$ finite, it is well known that the continuity of $g$ is equivalent to $\lim_{k \to \infty} \textrm{var}_{k}(g)=0$.

A non-negative measurable function $g$ is a {\it probability kernel}  (or simply a \emph{kernel}) if it satisfies
\begin{equation*}
\sum_{a\in A}g(xa)=1
\end{equation*}
for all $x\in \mathcal{X}_{-}$. This function has the same role as a transition matrix for a Markov chain.
As with transition matrices, we need to iterate $g$. So we define for all $n\geq 1$, the {\it iterated kernel} $g_n$ by 
\begin{equation}\label{eq:gn}
g_n(x, v_1^k)=\sum_{y\in \mathcal{X}}\prod_{i=1}^{n-1}g(xy_1^i)\prod_{j=1}^kg(xy_1^{n-1}v_{1}^{j})
\end{equation}
for all $k\geq 1$, $x\in \mathcal{X}_{-}$ and $v\in  \mathcal{X}$ (with the convention that $\prod_{i=l}^m=1$ if $l>m$.) Notice that with this definition, $g(xa)=g_1(x,a)$.\\

\noindent{\bf Stationary probability measures compatible with a kernel.}
Given a probability measure $\mu$ on $\X_-$ and a kernel $g$, we define  the probability measure $P^{\mu}$ on $\mathcal{X}$ such that
\[
P^{\mu}[\eta_{-\infty}^{0}\in B]=\mu[B]
\]
for all measurable $B\subset \mathcal{X}_-$ and for $n\geq 0$
\begin{equation*}\label{compa}
P^{\mu}[\eta_{n+1}=a\mid \eta_{-\infty}^{n}=x]=g(xa)
\end{equation*}
for every $a\in A$ and $\mu$-a.e.~$x$ in $\mathcal{X}_{-}$. A probability measure satisfying the above properties is said {\it compatible} with $g$. If $\mu=\delta_x$ for some $x\in \X_-$, we write $P^x$ instead of $P^{\delta_x}$. Finally, $P^{\mu}$ is {\it stationary} if it is translation invariant: $P^{\mu}\circ\sigma^{-1}=P^\mu$.

 When $A$ is finite, a compactness argument shows that if the kernel $g$ is continuous, at least one compatible stationary probability measure exists (see for example \cite{keane/1972}). When there is more than one stationary probability measure compatible with $g$, we say that there is a \emph{phase transition}, when there is one we usually say that there is \emph{uniqueness}.
The rate of decay of the sequence $(\textrm{var}_{k}(g))_{k\geq 1}$ is related to the uniqueness of a compatible stationary process.

\begin{defi}
A measurable function $g \in \mathcal{W}^{2}$ if for all $x, y \in \mathcal{X}_-$, we have
\begin{align*}
\sum_{k\geq 1} \sum_{a\in A}\Big( &g(xw_{1}^{k}a)-g(yw_{1}^{k}a)  \Big)^2<\infty,\;\; \text{for $P^x$-a.e.~w}.
\end{align*}
\end{defi}
%
%
%
%
%

It is also known that there is uniqueness if $g$ is continuous, strongly non-null ($\inf g>0$) and  belongs to $\mathcal{W}^{2}$ (\cite{johansson/oberg/2003}). A slightly more general positivity assumption that has been made in the literature is that of weak non-nullness $\sum_{a\in A}\inf_{x\in\mathcal X_-} $\\$g(xa)>0$. 
In this paper, we are interested in investigating the uniqueness question for \emph{null} kernels, which may not be weakly non-null. {Having in mind the case of $k$-step Markov chains, we already know that reducibility issues may appear depending on the ``location'' of the $0$'s in the transition matrix, leading to phase-transition in a rather trivial way. However, it is important to observe that the very notion of irreducibility has not been clearly defined, to the best of our knowledge, in the context of non-Markov chains that we consider here. Therefore, an important part of the problem is to define several quantities that will help us to correctly address these questions. } 

\section{Main results} \label{sec:results}
Before we can state our results, we need some further definitions. 

\vspace{0,3cm}

\noindent{\bf Further definitions.} The set $\mathcal D=\{x\in \X_-: \exists a\in A \,\text{ s.t.}\; g(xa)=0\}$ is the set of $0$'s of $g$. We partition it into 
$\mathcal D=\mathcal D^{(<\infty)}\cup\mathcal D^{(\infty)}$ with
\begin{equation}\label{eq:cyl_g}
\mathcal D^{(<\infty)}:=\cup_{i\ge1}\mathcal D^{(i)}:=\cup_{i\ge1}\{x\in \X_-: \exists\; a\in A \,\text{ s.t.}\; g(yx_{-i+1}^0a)=0\; \text{for all}\; y\in \X_-\}
\end{equation}
and
\begin{equation}\label{eq:isolated0}
\mathcal D^{(\infty)}:=\{x\in \X_-: \exists\; a\in A\, \text{ s.t.}\,\,g(xa)=0  \;\text{and}\; x\notin\mathcal D^{(<\infty)} \}.
\end{equation}
In words, $\mathcal D^{(<\infty)}$ is the set of {\it cylinders} of $0$'s of $g$, which can be identified by scanning the process locally, while $\mathcal D^{(\infty)}$ contains the remaining \emph{isolated} $0$'s of $g$ which can only be identified if we see the entire past of the process. Observe that $(\mathcal D^{(i)})_{i\geq 1}$ is a nondecreasing sequence of sets.

We will consider binary square matrices (possibly infinite), that is, matrices whose elements belong to $\{0,1\}$. In this case, the addition and product will always be the boolean ones ($1+1=1$, $0\times 1=0$, etc). With these operations, the set of binary square matrices of a given dimension is a boolean algebra.



\begin{defi}\label{def:MK}
Fix a kernel $g$. For any $K\in \N$ we define the matrices $\underline{M}_K=\underline{M}_K^{(g)}$ and $\overline{M}_K=\overline{M}_K^{(g)} :A^K\times A^K\rightarrow\{0,1\}$ by
$$\underline{M}_K(u,v)
=\left\{
\begin{array}{ll}
1, &\mbox{if}\phantom{*}(u,v)=(x_{-K+1}^{0},x_{-K+2}^{0}a)\;\;\text{and}\;\; g(yx_{-K+1}^{0}a)>0,\; \text{for all}\;\; y,\\
0, &\mbox{ otherwise},
\end{array}
\right.
$$
and
$$\overline{M}_K(u,v)
=\left\{
\begin{array}{ll}
0, &\mbox{if}\phantom{*}(u,v)=(x_{-K+1}^{0},x_{-K+2}^{0}a)\;\;\text{and}\;\; g(yx_{-K+1}^{0}a)=0,\; \text{for all}\;\; y,\\
1, &\mbox{ otherwise}.
\end{array}
\right.
$$
\end{defi}
Observe that $\underline{M}_K\leq \overline{M}_K$ (entrywise) for all $K\in \N$. 

\begin{defi}
\label{Deferi}
The pair of matrices $(\underline{M}_K,\overline{M}_K)$ is \emph{essentially row irreducible (e.r.i.)} for some $K\in \N$, if there exists a unique non empty set $\mathcal P_K\subset A^K$ such that
{\begin{enumerate}[(i)]
\item for all $i \in \mathcal P_K$ and $j\notin \mathcal P_K$, $\overline{M}_K(i,j)=0$;
\item for all $i\in \mathcal P_K$, there exists $n=n(i)$ such that $\underline{M}_K^n(i,j)=1$ for all $j\in \mathcal P_K$ ;
\item for any $x\in \X_-$ such that $x_{-K+1}^{0}\notin \mathcal{P}_K$ we have $P^x$-a.s. 
\[
T_{\mathcal P_K}:=\inf\{m\geq 1: \eta_{m-K+1}^{m}\in \mathcal{P}_K\}<\infty.
\]
\end{enumerate}}
\end{defi}


The concept of e.r.i. will be used to take care of the cylinder sets of $0$'s $\mathcal D^{(<\infty)}$. 

Let us now focus on the isolated $0$'s of $g$, $\mathcal D^{(\infty)}$, and define, for any $x\in \X_-$, the (possibly empty) set of finite strings that visit these $0$'s of $g$,
\begin{equation}\label{eq:Sx}
\mathcal S_x=\{u\in A^*: xu\in\mathcal D^{(\infty)}\}.
\end{equation}

For the reader's convenience we provide a compact reminder of the main notational definitions in Table \ref{tab:notation}.
We now state our main result.
\begin{theo}
\label{theo1}
Consider a kernel $g$ for which there exists $K\ge1$ such that
{\begin{enumerate}[(A)]
\item \label{ass1} the pair $(\underline{M}_{K}, \overline{M}_{K})$ is e.r.i.;
\item \label{ass2} for any $x \in \mathcal{D}^{(<\infty)}\setminus \mathcal{D}^{(K)}$, $x_{-K+1}^{0}\notin \mathcal{P}_K$; 
\item\label{ass3} for any $x \in \X_-$ such that $x_{-K+1}^{0}\in \mathcal{P}_K$ there exists $v_x\in A^*$ which is not substring of any element of $\mathcal S_x$ and such that $g_1(x,v_x)>0$. 
\end{enumerate}}  
If furthermore $\sqrt{g}\in \mathcal{W}^{2}$, then there exists at most one stationary probability measure $\mu$ compatible with $g$. When it exists, $\mu$ is weak Bernoulli.
\end{theo}

Let us briefly comment on the intuition behind assumptions~\eqref{ass1}--\eqref{ass3}. 
The kernel $g$ may vanish on certain pasts, and the structure of its zero set plays a 
central role in the theorem. First, we observe that Assumptions~\eqref{ass1} and \eqref{ass2} control the local zeros of $g$, that is, those in $\mathcal{D}^{(<\infty)}$, while Assumption~\eqref{ass3} addresses the isolated zeros of $g$, those in $\mathcal{D}^{(\infty)}$.
Now, let us be more precise. Assumption \eqref{ass1} requires that the pair 
$(\underline{M}_K, \overline{M}_K)$ is essentially row irreducible, which means that the 
chain behaves, at the level of $K$-step transitions, like an irreducible and aperiodic 
Markov chain on the ``good'' states $\mathcal{P}_K$ (if $\text{card}(A)=\infty$, the e.r.i.~condition is actually slightly stronger than irreducibility and aperiodicity), and that every past eventually 
reaches $\mathcal{P}_K$ with probability one. Assumption~\eqref{ass2} is a consistency 
condition: it ensures that the zeros coming from cylinders of length larger than $K$ do 
not conflict with the irreducibility structure encoded in $\mathcal{P}_K$, such that those pasts are automatically outside of $\mathcal{P}_K$ and will eventually be left. 
Finally, Assumption~\eqref{ass3} 
requires that, for any past in $\mathcal{P}_K$, there exists a finite string that 
avoids all these isolated zeros, ensuring that the process can escape their 
influence with positive probability.

\begin{remark}
It is elucidating to compare our result with the following theorem proved in  
\cite{johansson/oberg/pollicot/2007}.
\begin{theo*}[\cite{johansson/oberg/pollicot/2007}] If the kernel $g$ satisfies \begin{equation} \label{eq:JOP}
\sum_{k\geq 1} \sup_{x,y, w}\sum_{a\in A}\Big(\sqrt{g(xw_{1}^{k}a)}-\sqrt{g(yw_{1}^{k}a)}  \Big)^2<\infty,\;\; 
\end{equation}
then any two distinct ergodic measures compatible with $g$ are not locally comparable. 
\end{theo*}
Local comparability of ergodic measures $\mu$ and $\nu$ means that for all $n \geq 0$ we have $\mu|_{\mathcal{F}_{[0,n]}} \sim \nu|_{\mathcal{F}_{[0,n]}}$ (that is, $\mu|_{\mathcal{F}_{[0,n]}} \ll \nu|_{\mathcal{F}_{[0,n]}}$ and $\nu|_{\mathcal{F}_{[0,n]}} \ll \mu|_{\mathcal{F}_{[0,n]}}$). 

We first observe that \eqref{eq:JOP} implies $\sqrt{g} \in \mathcal{W}^{2}$. Also, if $\inf g > 0$, the local comparability is imediate. Therefore, the above theorem implies that if the kernel is strongly non-null (i.e.~$\inf g > 0$), and satisfies \eqref{eq:JOP}, there is at most one ergodic compatible measure. Beyond this special case, it not straightforward to obtain a criterion based on the kernel $g$ that allows us to decide when the compatible ergodic measures are comparable. Theorem \ref{theo1} provides such a criterion.
\end{remark}

\begin{remark}
We recall that the measure $\mu$ is called \emph{weak Bernoulli} or, equivalently, \emph{absolutely regular} or \emph{$\beta$-mixing} \citep{berbee/1986, bradley/2005}, if 
\begin{equation*}\lim_{n \rightarrow \infty}d_{TV}(\mu|_{\F^-}\otimes\mu|_{\F_{[n, \infty)}}, \mu|_{\F^-\otimes\F_{[n, \infty)}}) = 0,
\end{equation*}
where $d_{TV}$ is the total variation distance. If $\mu$ is weak Bernoulli and the alphabet is finite, we have, in particular, that $\mu$ is isomorphic to a Bernoulli shift.
\end{remark}
\begin{remark}
The assumption \eqref{ass1} accounts for the local $0$'s of the kernel. It ensures some kind of irreducibility and aperiodicity. Since the kernel has infinite dependence, even in a finite alphabet it is possible that a string outside of $\mathcal P_K$ never visits $\mathcal P_K$ with positive probability, that is why we impose condition (iii) in Definition \ref{Deferi}.
 The pair of assumptions \eqref{ass2}--\eqref{ass3} accounts for those $0$'s which come from cylinders larger than $K$ (assumption \eqref{ass2}) and infinity (assumption \eqref{ass3}), so in some sense, nonlocal. 
\end{remark}

Theorem \ref{theo1} will be proved in Section~\ref{sec:proofs}. Notice that neither continuity of $g$ nor finiteness of $A$ is assumed, so existence is not granted, but adding these assumptions we have the following consequences, the first of which is direct.

\begin{coro}
\label{coro1}
Suppose $A$ finite and $g$ continuous. Then, under the assumptions of Theorem \ref{theo1}, there exists exactly one stationary probability measure compatible with~$g$.
\end{coro}
When we are given $g$, it is usually straightforward to check whether it belongs to $\mathcal{W}^{2}$ or not. However, for null kernels, it is not always the case that $g\in \mathcal{W}^{2}\Rightarrow\sqrt g\in \mathcal{W}^{2}$, and when it is the case, some further calculations are needed (see Examples \ref{example3} and \ref{example4}). At the cost of some restriction on $\mathcal D$, the next corollary uses a condition directly on $g$ instead of $\sqrt{g}$ (see Examples \ref{sec:markov} and \ref{NMRWDG} for applications).
\begin{coro}
\label{coro2}
Suppose $A$ finite, $g$ continuous, ${g}\in \mathcal{W}^{2}$ and $\mathcal{D}^{(\infty)}=\emptyset$. Then under assumptions \eqref{ass1}--\eqref{ass2} of Theorem \ref{theo1}, there exists exactly one stationary probability measure compatible with~$g$.
\end{coro}

\begin{table}[ht]
\centering
\renewcommand{\arraystretch}{1.3}
\begin{tabular}{lll}
\hline
\textbf{Notation} & \textbf{Def.} &  \textbf{Description} \\
\hline
$D^{(<\infty)}$ &   Disp. \eqref{eq:cyl_g}&  \begin{tabular}[t]{@{}c@{}}Cylinder $0$'s of $g$: identifiable from a finite portion \vspace{-0,2cm}\\ of the past. \end{tabular}    \\
$D^{(\infty)}$  &  Disp.  \eqref{eq:isolated0}& Isolated $0$'s of $g$: require the entire  past to be detected. \\
$\mathcal S_x$           &  Disp. \eqref{eq:Sx}& Finite strings $u \in A^*$ such that $xu \in D^{(\infty)}$. \\
$\overline{M}_K$  &  Def. \ref{def:MK}&\begin{tabular}[t]{@{}c@{}} Binary $A^K\!\times\!A^K$ matrix where $0$'s correspond to uniformly \vspace{-0,2cm}\\ forbidden transitions of length $K$. \end{tabular} \\
$\underline{M}_K$ &   Def. \ref{def:MK}& \begin{tabular}[t]{@{}c@{}}Binary $A^K\!\times\!A^K$ matrix where $1$'s correspond to uniformly     \vspace{-0,2cm}\\ allowed transitions of length $K$.\end{tabular} \\
$\mathcal{P}_K$ &  Def. \ref{Deferi}& Essential class of $A^K$ when $(\overline{M}_K, \underline{M}_K)$ is e.r.i. \\
\hline
\end{tabular}
\vspace{1mm}
\caption{Key notation introduced in Section~3.}
\label{tab:notation}
\end{table}

\section{Applications on some example processes} \label{sec:examples}

To illustrate the above results, we give six examples, two for each result, starting with Corollary \ref{coro2}, having most restrictive assumptions, and ending with Theorem \ref{theo1}, having less restrictive assumptions.

\subsection{$k$-step Markov chains}\label{sec:markov}
 We call a stochastic chain with kernel $g$ a $k$-step Markov chain if $g$ satisfies $\text{var}_k(g)=0$ for some $k\geq 1$.
 Thus, we have $\overline M_K=\underline M_K$ for $K\ge k$. Existence and uniqueness are well-studied questions in this context. For example, on finite alphabet, existence is always granted and uniqueness is equivalent to $\overline M_k$ having a unique closed class. To apply Corollary \ref{coro2} (with $K=k$), we just have to check that $\overline M_k$ has a unique closed class and that this class is aperiodic (see (ii) in Definition \ref{Deferi}). This is equivalent to verifying that the $1$-step Markov kernel on alphabet $A^k$ defined by $p(x^{-1}_{-k}, x^{-1}_{-k+1}a) = g(x^{-1}_{-k},a)$ has a unique recurrent class and that this class is aperiodic. This last condition is a well known sufficient condition for uniqueness. In this example, the key nontrivial assumptions are~\eqref{ass1} when $K=k$, and ~\eqref{ass1}--\eqref{ass2} when $K<k$.
\vspace{0,3cm}

We now turn to non-Markov examples, our main interest. The first is an application of Corollary \ref{coro2}. 

\subsection{Non-Markovian random walks on directed graph}
\label{NMRWDG}
Suppose for this example that $A$ is finite. It is common to see Markov chains  as random walks on the graph of communications resulting from their transition matrix. Here, if  $\underline M_1=\overline M_1$, we can see any chain  compatible with $g$ as a non-Markovian random walk on the oriented graph $\mathcal G=(A,\mathcal E)$ with
\[
\mathcal E=\{(a,b)\in A^2:\underline M_1(a,b)=1\}. 
\]
In that context \cite{desantis/piccioni/2010} (see Example 4 and Corollary 2 therein) state that if $\underline M_1$ is irreducible, aperiodic and $\sum_{n\ge1}\prod_{i=1}^na_i=\infty$ (see \cite[p324]{desantis/piccioni/2010} or \cite{comets_fernandez_ferrari_2002} for the definition of the $a_i$'s), then it is possible to perfectly simulate the unique compatible stationary measure. In view of our Corollary \ref{coro2}, since $A$ is finite and $\underline M_1=\overline M_1$ is irreducible and aperiodic, it automatically follows that $\mathcal{D}^{(\infty)}=\emptyset$ and \eqref{ass1}--\eqref{ass2} are satisfied. Thus if we further assume that $\text{var}_k(g)\rightarrow0$ and $g\in\mathcal W^2$ there exists a unique stationary chain, although we cannot guarantee that there exists a perfect simulation algorithm for it.  

In order to compare with their assumption, observe that $a_i=O(1-\text{var}_i(g)),i\ge1$ on finite alphabets.
Therefore the assumption of \cite[Corollary 2]{desantis/piccioni/2010} that $\sum_{n\ge1}\prod_{i=1}^n a_i=\infty$ is stronger than $\text{var}_k(g)\rightarrow0$ and $g\in\mathcal W^2$ together on finite alphabets. Hence, the nontrivial assumptions of Corollary \ref{coro2} for this example are the square summability condition $g\in\mathcal{W}^2$ and the continuity of $g$.
\vspace{0,3cm}

The next example is an application of Corollary \ref{coro1}. We will use the notation $\underline{a}=\ldots a a\in \X_-$ for the past obtained concatenating infinitely many times the same symbol $a$. 

\subsection{Sparse chain with null cylinder larger than $K$}\label{example3}

We consider $A=\{0,1\}$ and a kernel $g$ defined by
$g(x010)=g(x100)=g(x10000)=g(\underline{0})=0$ for all $x\in \X_-$. Also, we define for all $j\geq 0$ the probability kernel
\begin{equation*}
p_j(1\mid x_{-j}^0)=\frac{2}{3}\cdot {\bf 1}\Big\{\frac{1}{j+1}\sum_{l=0}^jx_{-l}\geq 1/2\Big\}+\frac{1}{3}\cdot {\bf 1}\Big\{\frac{1}{j+1}\sum_{l=0}^jx_{-l}<1/2\Big\}.
\end{equation*}
Then, we complete the definition of $g$ with
\[g(x110)=g(x1000)=\sum_{j=0}^{\infty}\frac{1}{2^{j+1}} p_j(0\mid x_{-j}^0) 
\]
and for $k\geq 3$,
\[g(x1\underline{0}_{-k}^00)=\frac{1}{\sqrt{k}} \sum_{j=0}^{\infty}\frac{1}{2^{j+1}} p_j(0\mid x_{-j}^0),
 \]
for all $x\in \X_-$.
Now, observe that 
\[
   \underline M_2= \begin{bmatrix}
    0 & 0 &1 & 0\\  0 & 0 & 1& 0\\0 & 0 &0 & 1\\0 & 1 & 0& 1
    \end{bmatrix}%
,\;\;\;\;\;\;\overline M_2   = \begin{bmatrix}
     1 & 0 & 1& 0\\  0 & 0 & 1& 0\\0 &0 & 0& 1\\0 & 1 & 0& 1
    \end{bmatrix},
\]
using $00<10<01<11$ to order columns and rows of the matrices. We observe that the pair $(\underline M_2, \overline M_2)$ is e.r.i.~and $\mathcal P_2=\{10, 01, 11\}$. Also, $\mathcal{D}^{(\infty)}=\{\underline{0}\}$ and $\mathcal{D}^{(<\infty)}\setminus \mathcal{D}^{(2)}=\{x1000,\;x\in \X_-\}$. We can easily check that assumptions \eqref{ass1}--\eqref{ass3} of Theorem \ref{theo1} are satisfied.

Let us also show that $\sqrt{g}\in \mathcal{W}^2$. For all $x\in \X_-$, we have for $P^x$-a.e.~$w$ that $w_{k-1}^k\in \{10,01,11\}$ for large enough $k$. On the other hand, observe that for $k\geq 2$
\begin{align*}
\sum_{a\in A}\Big(\sqrt{g(xw_{1}^{k}a)}-\sqrt{g(yw_{1}^{k}a)}\Big)^2=0,\;\;\;\text{for}\;w_{k-1}^k\in \{10,01\},
\end{align*}
and for $w_{k-1}^k=11$,
\begin{align*}
\sum_{a\in A}\left(\sqrt{g(xw_{1}^{k}a)}-\sqrt{g(yw_{1}^{k}a)}\right)^2&\leq \sum_{a\in A}\frac{(g(xw_{1}^{k}a)-g(yw_{1}^{k}a))^2}{g(xw_{1}^{k}a)}\nonumber\\
&\leq \frac{2}{3}\left(\sum_{j\geq k-2}\frac{1}{2^{j+1}}\right)^2\nonumber\\
&=\frac{2}{3}\frac{1}{4^k}.
\end{align*}
Gathering all these facts, we deduce that $\sqrt{g}\in \mathcal{W}^2$. Since $g$ is also continuous, we can apply Corollary \ref{coro1} to obtain existence and uniqueness of the stationary compatible probability measure. 

Note that, in this example, all the assumptions of Corollary \ref{coro1} had to be checked. In particular, assumption \eqref{ass2} has to be verified since $\mathcal{D}^{(<\infty)}\setminus \mathcal{D}^{(2)}$ is not empty.
\vspace{0,3cm}

Here is another application of Corollary \ref{coro1}. 
\subsection{Linear autoregressive model}\label{example4}
Let $A=\{-1,1\}$ and consider a sequence of real numbers $\alpha_i>0$, $i\ge1$ s.t. $\sum_i\alpha_i=1/2$, and put for any $x\in \mathcal X_-$
\[
g(x)=\frac{1}{2}+\epsilon x_0\sum_{i\ge1}\alpha_ix_{-i}
\]
where $\epsilon\in\{-1,1\}$.  
 
Notice that if $\epsilon=1$, then assumption \eqref{ass3} of Theorem \ref{theo1} is not fulfilled, because $g_n(\underline{1},1)=1$ for any $n$ (and the same for $\underline{-1}$). This case is similar to the case of reducible Markov chain, since $\delta_{\{1\}^{\Z}},\delta_{\{-1\}^{\Z}}$ are both trivial stationary measures. 

We therefore assume $\epsilon=-1$. Then for all $K\ge1$, $\mathcal D^{(K)}=\emptyset$, the matrix $\overline M_K= 1$ (where 1 is the all-ones matrix) while $\underline{M}_K(\underline{-1}_{-K+1}^0,\underline{-1}_{-K+1}^0)=\underline{M}_K(\underline{1}_{-K+1}^0,\underline{1}_{-K+1}^0)=0$. It follows that $(\underline M_2,\overline M_2)$ is e.r.i with $\mathcal P_2=A^2$ and assumptions \eqref{ass1}--\eqref{ass2} hold. Finally, notice that  $\mathcal D^{(\infty)}=\{\underline{-1},\underline{1}\}$ and taking $v_{\underline{-1}}=1,v_{\underline{1}}=-1$ we see that \eqref{ass3} also holds.

Let us now check that if $\sum_{k\geq 1}(\sum_{i\geq k+1}\alpha_i)^2<\infty$ then $\sqrt{g}\in \mathcal{W}^2$. For any $x,y,w\in \mathcal{X}_-$ and $a\in A$, we have
\begin{align*}
\left(\sqrt{g(xw_{1}^{k}a)}-\sqrt{g(yw_{1}^{k}a)}\right)^2&=\left(\frac{{g(xw_{1}^{k}a)}-{g(yw_{1}^{k}a)}}{\sqrt{g(xw_{1}^{k}a)}+\sqrt{g(yw_{1}^{k}a)}}\right)^2\nonumber\\
&\le \Bigg(\frac{a\sum_{i\ge k+1}\alpha_i(y_{-i}-x_{-i})}{\sqrt{g(xw_{1}^{k}a)}}\Bigg)^2\wedge 1.
\end{align*}
Since the model is binary and symmetric, in order to check that $\sqrt{g}\in\mathcal W^2$ it is enough to check that for all $x$, 
\begin{equation}
\label{sumyui}
\sum_k\Bigg[\frac{\left(\sum_{i\ge k+1}\alpha_i\right)^2}{g(xw_{1}^{k}1)}\wedge 1\Bigg]<\infty,
\end{equation}
for $P^x$ almost every $w$. For this, observe that 
\begin{align*}
g(xw_{1}^{k}1)=\frac{1}{2}-\sum_{i=1}^k\alpha_iw_{i}-\sum_{i=k+1}^{\infty}\alpha_i x_{-i+k+1}.
\end{align*}
Since the $\alpha_i$ are positive numbers, we deduce for all $x$, 
$$P^x\left[\lim_{k\to \infty}\sum_{i=1}^k\alpha_iw_{i}<1/2\right]=1,$$
since $P^x[w_{i}=1, \forall i\geq 1]=0$. 


Now, fix $x$ and let $w\in \{\lim_{k\to \infty}\sum_{i=1}^k\alpha_iw_{i}<1/2\}$. We deduce that there exists $\delta=\delta(x,w)>0$ such that, for all sufficiently large $k$,
\begin{equation*}
\sum_{i=1}^k\alpha_iw_{i}\leq \frac{1}{2}-\delta
\end{equation*}
and 
\begin{equation*}
\sum_{i=k+1}^{\infty}\alpha_i x_{-i+k+1}\leq \frac{\delta}{2}.
\end{equation*}
Therefore, we obtain that
\begin{align*}
g(xw_{1}^{k}1)\geq\frac{1}{2}-\Big(\frac{1}{2}-\delta\Big)-\frac{\delta}{2}\geq \frac{\delta}{2}
\end{align*}
and finally, using (\ref{sumyui}),
\begin{align*}
\sum_{k\geq 1}\left(\sqrt{g(xw_{1}^{k}a)}-\sqrt{g(yw_{1}^{k}a)}\right)^2<\infty
\end{align*}
$P^x$-a.s.~when $\sum_k\left(\sum_{i\ge k+1}\alpha_i\right)^2<\infty$. 
 Observe that for this model, the condition $\sum_{k\geq 1}(\sum_{i\geq k+1}\alpha_i)^2<\infty$ is equivalent to $(\text{var}_k(g))_k\in\mathcal \ell^2$. Hence, if $(\text{var}_k(g))_k\in\mathcal \ell^2$, we can apply Corollary \ref{coro1}, to obtain the existence of a unique stationary probability measure compatible with $g$.
A similar example is considered by \cite{desantis/piccioni/2010} with the special sequence $\alpha_i=\frac{1}{2^{i+1}},i\ge1$, to obtain a perfect simulation algorithm, a stronger property that in particular implies existence and uniqueness. 

In this example, the key nontrivial assumption is~\eqref{ass3}, which fails 
when $\epsilon=1$ and illustrates its necessity to obtain uniqueness of the invariant compatible probability. We also obtain a sufficient condition to obtain $\sqrt{g}\in\mathcal{W}^2$.
\vspace{0,3cm}

The following binary example shows an application of Theorem \ref{theo1} when the kernel can be discontinuous.

\subsection{The renewal process}
Let $A=\{0,1\}$ and define for any $x\in\mathcal X$ and $n\in\mathbb Z$ the function $\ell(x_{-\infty}^n)=\min\{i\ge0:x_{n-i}=1\}$, counting the number of zeros at the end of $x_{-\infty}^n$ (with the convention that $\ell(\underline 0)=+\infty$). Consider also a $[0,1]$-valued sequence $q_i$, $i\in \mathbb Z_+\cup\{+\infty\}$. For any $x\in\mathcal X_-$, let $g(x1)=q_{\ell(x)}$. In the following, we assume that $q_i\in(0,1)$ for any  $i\in \Z_+$.

This example is enlightening to understand Theorem \ref{theo1}. This is mainly due to two facts. First, the complete picture concerning existence and uniqueness, with respect to the choice of the sequence $q_i$, is known in the literature (see for instance \cite{cenac/chauvin/paccaut/pouyanne/2010}):
\begin{enumerate}
\item If $q_\infty=0$
\begin{enumerate}
\item with $\sum\prod(1-q_i)=\infty$: there is a unique stationary measure (the trivial measure $\delta_{\{0\}^\mathbb Z}$);
\item with $\sum\prod(1-q_i)<\infty$: there are more than one stationary measure ($\delta_{\{0\}^\mathbb Z}$ and the renewal measure with infinitely many $1$'s are stationary in that case).
\end{enumerate}
\item If $q_\infty>0$
\begin{enumerate}
\item with $\sum\prod(1-q_i)=\infty$: there exists no stationary measure;
\item with $\sum\prod(1-q_i)<\infty$: there is a unique stationary measure (the renewal measure with infinitely many $1$'s).
\end{enumerate}
\end{enumerate}
Second, as already noticed in \cite{GGT2018}, it exemplifies the role of the $P^x$-a.s. in the definition of $g\in\mathcal W^2$. To see that, note that if for some $1\leq j\leq k$, $w_j=1$, then $\sqrt{g(xw_{1}^{k}a)}=\sqrt{g(yw_{1}^{k}a)}$ for any $x\in \mathcal X_-$ and $a\in A$. Suppose then that for some $x\in\mathcal X_-$ we have that $P^x(w_i=0,i\ge1)=0$, then for $P^x$-a.e. $w\in \mathcal X_+$, there exists $N_w\geq 1$ such that 
\[
\sum_{k\geq 1} \sum_{a\in A}\Big( \sqrt{g(xw_{1}^{k}a)}-\sqrt{g(yw_{1}^{k}a)}  \Big)^2=2\sum_{k=1}^{N_w}\Big( \sqrt{g(xw_{1}^{k}a)}-\sqrt{g(yw_{1}^{k}a)}  \Big)^2<\infty.
\]
On the other hand, a simple calculation gives for any $x\neq \underline{0}$ 
\[P^x[w_i=0,i\ge1]=\lim_k\prod_{l=\ell(x)}^{\ell(x)+k}(1-q_l)=0 \,\Leftrightarrow\,\prod_{l=0}^\infty(1-q_l)=0. \]
Hence we may have $\sqrt{g}\in\mathcal W^2$ although $g$ is not continuous.

Regarding Theorem \ref{theo1}, 
\begin{itemize}
\item If $q_\infty=0$ \\
Assumption \eqref{ass1} is not satisfied for any $K\geq 1$ and it follows that Theorem \ref{theo1} does not apply. As we saw in cases (1a) and (1b) above, either there were many stationary measures or only $\delta_{\{0\}^\mathbb Z}$. 
Despite uniqueness, case (1b) is kind of  ``pathological'' since the invariant measure 
is fully concentrated on the unique absorbing point $\underline{0}$ (observe that assumption \eqref{ass3} is not verified for $x=\underline{0}$).
\item If $q_\infty>0$\\
 Assumption \eqref{ass1} is satisfied with $K=1$ and $\mathcal P=A$. Now, if $\prod_l(1-q_l)=0$ we have $\sqrt{g}\in\mathcal W^2$ as discussed above, we conclude that Theorem \ref{theo1} applies, and therefore that there exists at most one stationary measure. Indeed, we are in case (2) above, and either there is no stationary measure, as in (2a), or there is exactly one stationary measure, as in (2b). If otherwise $\prod_l(1-q_l)>0$, Theorem \ref{theo1} does not apply and indeed in that case we are in case (2b) and no stationary measure exists.
\end{itemize}

In this example the nontrivial assumption is~\eqref{ass1}: 
when $q_\infty=0$ it fails and Theorem~\ref{theo1} does not apply, 
correctly reflecting the possibility of phase transition or pathological 
uniqueness; when $q_\infty>0$, assumption~\eqref{ass1} holds and the 
theorem applies whenever $\prod_l(1-q_l)=0$.
\vspace{3mm}

For our last example also, the corollaries do not apply as the alphabet is infinite and we need to invoke Theorem \ref{theo1}. 
\subsection{Example with infinite alphabet}
Let $A=\Z_+$ and consider  a positive constant $\gamma$ and a sequence of real numbers $(\beta_i)_{i\ge 0}$, such that $\sum_{i\geq 0}|\beta_i| <\infty$. For any $x\in \mathcal X_-$, define
\begin{equation*}
v(x)=\exp\left\{ -\sum_{j\geq 0}\beta_j(x_{-j}\wedge \gamma) \right\}.   
\end{equation*}
The kernel is now defined, for any $x\in \mathcal X_-$, by $g_1(x,a)=e^{-v(x)}\frac{v(x)^a}{a!}$ if $x_{0}=0$, and $g_1(x,0)=1$ otherwise.
We easily check that $\underline{M}_1(0,j)=\underline{M}_1(j,0)=1$ for all $j\in \Z_+$, and that all other entries vanish.
Moreover $\underline{M}_1=\overline{M}_1$, $\mathcal{P}_1=A$, $\mathcal{D}^{(<\infty)}=\mathcal{D}^{(1)}=\{x\in \X_-:x_0\neq 0\}$ and $\mathcal{D}^{(\infty)}=\emptyset$. From here it is not hard to see that the assumptions \eqref{ass1}--\eqref{ass3} of Theorem \ref{theo1} are satisfied for $K=1$. Now, we show that if $\sum_{k\geq 1}(\sum_{j\geq k+1}|\beta_j|)^2<\infty$ then $\sqrt{g}\in \mathcal{W}^2$ and therefore by Theorem \ref{theo1} there exists at most one invariant compatible probability measure.
We recall that if $w_k\neq 0$, $g(xw_1^k0)=1$ and therefore
\begin{align*}
 \sum_{a\in A}\left(\sqrt{g(xw_1^ka)}-\sqrt{g(yw_1^ka)}  \right)^2=0.
\end{align*}
Hence, let us consider $w_k=0$. In this case, after some elementary computations we obtain for $k\geq 1$,
\begin{align*}
 \sum_{a\in A}\left(\sqrt{g(xw_1^ka)}-\sqrt{g(yw_1^ka)}\right)^2
 &\leq \sum_{a\in A}\frac{\left(g(xw_1^ka)-g(yw_1^ka)  \right)^2}{g(xw_1^ka)}\nonumber\\
 &=\sum_{a\in A}e^{-v(xw_1^k)}\frac{v(xw_1^k)^a}{a!}\nonumber\\
&\phantom{*****} \times \left( 1-\left(\frac{v(yw_1^k)}{v(xw_1^k)}\right)^ae^{-(v(yw_1^k)-v(xw_1^k))} \right)^2\nonumber\\
 &=\exp\left\{\frac{(v(xw_1^k)-v(yw_1^k))^2}{v(xw_1^k)}\right\}-1 \nonumber\\
 & \leq \exp\left\{c(v(xw_1^k)-v(yw_1^k))^2\right\}-1
\end{align*}
for some positive constant $c$ (to obtain the second equality just expand the square and sum the series). Then, observe that
\begin{align*}
|v(xw_1^k)-v(yw_1^k)|\leq 2\gamma\sum_{j\geq k+1}|\beta_j|
\end{align*} and as $k\to \infty$,
\begin{align*}
\exp\left\{c(v(xw_1^k)-v(yw_1^k))^2\right\}-1\sim c(v(xw_1^k)-v(yw_1^k))^2.
\end{align*}
By the equivalence criterion for non-negative series, we conclude that 
\begin{equation}
\label{sumcondinf}
\sum_{k\geq 1}\left(\sum_{j\geq k+1}|\beta_j|\right)^2<\infty
\end{equation}
implies $\sqrt{g}\in \mathcal{W}^2$. In this example, we can actually apply Theorem 5.1 of \cite{johansson/oberg/pollicot/2007} to obtain the existence of the invariant measure. Indeed, using the notations of \cite{johansson/oberg/pollicot/2007}, we can take $K=\exp{\exp(\gamma \sum_{i\geq 0}|\beta_i|)}$ and $\pi$ equal to the Poisson law with parameter $e^{\gamma \sum_{i\geq 0}|\beta_i|}$. Combining with our Theorem \ref{theo1}, we conclude that there exists a unique stationary compatible probability.

In this example, the corollaries do not apply due to the infinite alphabet, so the full strength of Theorem~\ref{theo1} is needed. The key nontrivial 
assumption is $\sqrt{g}\in\mathcal{W}^2$, whose verification is implied by the 
summability condition (\ref{sumcondinf}).

\section{Proofs of the results}\label{sec:proofs}

We start this section with some technical lemmas.

\begin{lemma}
\label{rowirredlem}
Let $(\underline{M}_K,\overline{M}_K)$ be e.r.i.~for some $K\in \N$. Then for all $i\in \mathcal P_K$ and all $m\geq n$ (where $n$ is from Definition \ref{Deferi} (ii)) we have $\underline{M}_K^m(i,j)=1$ for all $j\in \mathcal P_K$.
\end{lemma}
\noindent
\begin{proof}
Consider $i$ and $n=n(i)$ such that $\underline{M}_K^n(i,j)=1$ for all $j\in \mathcal P_K$.
It is enough to show that the above property remains true for $m=n+1$. By Definition \ref{Deferi} (ii), for all $l\in \mathcal P_K$ there must exist some $k\in \mathcal P_K$ such that $\underline{M}_K(k,l)=1$. Thus, we obtain that $\underline{M}_K^m(i,j)=1$ for all $j\in \mathcal P_K$.
\end{proof}
\begin{lemma}
\label{lemlac}
 Let $(\underline{M}_K,\overline{M}_K)$ be e.r.i.~for some $K\in \N$. For all $x\in \X_-$ such that $x_{-K+1}^0 \in \mathcal P_K$, there exists $n_x$ such that for all $n\geq n_x$, $g_n(x, v)>0$ for $v\in \mathcal P_K$ and $g_n(x, v)=0$ for $v\notin \mathcal P_K$.
\end{lemma}
\noindent
\begin{proof}
 Fix $x\in\X_-$ such that $u_x:=x_{-K+1}^0 \in \mathcal P_K$. Since the pair $(\underline M_{K},\overline M_{K})$ is e.r.i, by Lemma \ref{rowirredlem}, there exists $\underline n$ such that 
for any $n\ge \underline n$, $\underline{M}_{K}^n(u_x,v)=1$ for any $v\in \mathcal P_K$.

Let us define for $w\in \mathcal{X}_-$, the binary variables
\[
\delta(w)=\left\{
\begin{array}{ll}
1, &\mbox{if}\phantom{*}g(w)>0,\\
0, &\mbox{ otherwise}
\end{array}
\right.
\]
and observe that $g_n(x, z_1^{K})>0$ for $n\geq 2$ (see \eqref{eq:gn}),  if and only if
\[
\sum_{w}\prod_{i=1}^{n-1}\delta(xw_1^i)\prod_{j=1}^{K}\delta(xw_1^{n-1}z_{1}^{j})=1,
\]
we use here the boolean operations when multiplying and  adding binary variables.

Consider without loss of generality that $\underline n\ge K+2$. By definition of $\underline{M}_K$, we have for any $z_1^{K}\in\mathcal P_K$ and $n\ge \underline n$  that
\begin{align}\label{eq:geq}
\sum_{w}\prod_{i=1}^{n-1}\delta(xw_1^i)&\prod_{j=1}^{K} \delta(xw_1^{n-1}z_{1}^{j})\nonumber\\
&\ge \sum_{w} \prod_{l=1}^{{K}} \underline{M}_{K}(x_{l-{K}}^{0}w_1^{l-1},x_{l-{K}+1}^{0}w_1^l)\prod_{l={K}+1}^{n-1} \underline{M}_{K}( w_{l-{K}}^{l-1}, w_{l-{K}+1}^l) \nonumber\\
&\phantom{****}\times \prod_{j=1}^{K}\underline{M}_{K}(w_{n+j-{K}-1}^{n-1}z_1^{j-1}, w_{n+j-{K}}^{n-1}z_1^j)\nonumber\\
&=\underline{M}_{K}^{n+K}(u_x,z_1^{K})=1.
\end{align}
 We can apply the same reasoning once again with $\overline M_K$ instead of $\underline M_K$. For this, we just have to replace $\underline n$ by $1$ and $\geq$ by $\leq$ in \eqref{eq:geq} to get that $g_n(x, z_1^{K})\leq \overline{M}_{K}^{n+K}(u_x,z_1^{K})=0$ for  $n\geq 1$ and $z_1^K\notin  \mathcal P_K$. Finally, take $n_x=\underline n$ to obtain the desired result.
\end{proof}
\begin{prop}
\label{proplac1}
Under assumptions \eqref{ass1}--\eqref{ass3} of Theorem \ref{theo1}, for all  $x, y\in \X_-$ such that $x_{-K+1}^0$ and $y_{-K+1}^0 \in \mathcal P_K$, there exists $n=n(x,y)$ such that 
$P^{x}|_{\F_{[n,n+k]}} \sim P^{y}|_{\F_{[n,n+k]}}$ for all $k\geq 1$.
\end{prop}
\begin{proof}
Fix $x,y\in \X^-$ such that $x_{-K+1}^0$ and $y_{-K+1}^0 \in \mathcal P_K$. By assumption \eqref{ass3} of Theorem \ref{theo1}, there exist $v_x$ and $v_y$ such that $g_{1}(x,v_x)$ and $g_{1}(y,v_y)$ are strictly positive and $v_x$, $v_y$  are not substrings of any elements of $\mathcal S_x$ and $\mathcal S_y$ respectively. Let $R:=|v_x|\vee |v_y|$. Now, thanks to the last affirmation, there exist two strings $w_x$ and $w_y$ of size $R$ which contain $v_x$ and $v_y$ respectively as substrings and such that $g_1(x,w_x)>0$ and $g_1(y,w_y)>0$. Observe that we necessarily have that $(xw_x)_{-K+1}^0, (yw_y)_{-K+1}^0\in \mathcal P_K$ since $x_{-K+1}^0, y_{-K+1}^0 \in \mathcal P_K$. By Lemma \ref{lemlac}, there exists $\hat{m}$ such that for all $m\geq \hat{m}$,  $g_{m}(xw_x,z_1^K)$ and  $g_{m}(yw_y,z_1^K)$ are both strictly positive for all $z_1^K\in \mathcal{P}_K$ and $g_{m}(xw_x,z_1^K)= g_{m}(yw_y,z_1^K)=0$, for $z_1^K\notin \mathcal{P}_K$. Again, by Lemma \ref{lemlac}, there exists $\tilde{m}$ such that for all $m\geq \tilde{m}$, $g_m(x,z_1^K)$ and  $g_m(y,z_1^K)$ are both strictly positive for all $z_1^K\in \mathcal{P}_K$ and $g_m(x,z_1^K)= g_m(y,z_1^K)=0$, for $z_1^K\notin \mathcal{P}_K$. In the rest of the proof we take $m=\hat{m}\vee(\tilde{m}+ R)$. 

Now, consider $g_m(u, z_1^{K+k})$
for $u\in \{x,y\}$ and $z_1^{K+k}\in A^{K+k}$, $k\geq 1$. We will show that $g_m(x, z_1^{K+k})>0$ if and only if $g_m(y, z_1^{K+k})>0$ to prove the proposition.

First, observe that if $z_1^K\notin \mathcal P_K$, $g_m(x, z_1^{K+k})=g_m(y, z_1^{K+k})=0,k\ge1$.
On the other hand, observe that for $u\in \{x,y\}$ and $z_1^{K+k}\in A^{K+k}$, $k\geq 1$,
\begin{align*}
	g_m(u, z_1^{K+k})=g_1(u,w_u)g_{m-R}(uw_u,z_1^{K+k})
	+\sum_{w\neq w_u}g_1(u,w)g_{m-R}(uw,z_1^{K+k}).
\end{align*}
Since we are only interested in the positivity or nullness of $g_m(u, z_1^{K+k})$, we define
\[
\delta_n(u,v)=\left\{
\begin{array}{ll}
	1, &\mbox{if}\phantom{*}g_n(u,v)>0,\\
	0, &\mbox{ otherwise}.
\end{array}
\right.
\]
Thus, if $z_1^K\in \mathcal P_K$, we have by definition of $\overline{M}_K$ and assumption \eqref{ass2},
\begin{align*}
	\delta_m(u, z_1^{K+k})&=\delta_1(u,w_u)\delta_{m-R}(uw_u,z_1^{K+k})
	+\sum_{w\neq w_u}\delta_1(u,w)\delta_{m-R}(uw,z_1^{K+k})\nonumber\\
	&=\delta_{m-R}(uw_u,z_1^{K+k})
	+\sum_{w\neq w_u}\delta_1(u,w)\delta_{m-R}(uw,z_1^{K+k})\nonumber\\
	&=\delta_{m-R}(uw_u,z_1^K)\prod_{j=1}^{k}\overline{M}_K(z_j^{K+j-1},z_{j+1}^{K+j})\nonumber\\
	&\phantom{****}+\sum_{w\neq w_u}\delta_1(u,w)\delta_{m-R}(uw,z_1^{K+k})\nonumber\\
	&=\prod_{j=1}^{k}\overline{M}_K(z_j^{K+j-1},z_{j+1}^{K+j})+\sum_{w\neq w_u}\delta_1(u,w)\delta_{m-R}(uw,z_1^{K+k}).
\end{align*}
If $\overline{M}_K(z_j^{K+j-1},z_{j+1}^{K+j})=1$ for all $j\in\{1,\dots, k\}$, then $\delta_m(u, z_1^{K+k})=1$ for $u\in \{x,y\}$. Now if $\overline{M}_K(z_j^{K+j-1},z_{j+1}^{K+j})=0$, for some $j\in\{1,\dots, k\}$, it remains to show that $\delta_m(u, z_1^{K+k})=0$ for $u\in \{x,y\}$.
To this end, observe that by definition of $\overline{M}_K$ for all $w$, we have
\begin{equation*}
	\delta_{m-R}(uw,z_1^{K+k})\leq \delta_{m-R}(uw,z_1^K)\prod_{j=1}^{k}\overline{M}_K(z_j^{K+j-1},z_{j+1}^{K+j}).
\end{equation*}
Therefore, if $\overline{M}_K(z_j^{K+j-1},z_{j+1}^{K+j})=0$ for some $j\in\{1,\dots, k\}$ we obtain $\delta_m(u, z_1^{K+k})=0$ for $u\in\{x,y\}$.
By the former analysis, we obtain that $P^{x}|_{\F_{[m,m+k]}} \sim P^{y}|_{\F_{[m,m+k]}}$ for all $k\geq 1$.
\end{proof}
\begin{lemma}
\label{lem3}
Let $x\in \X$ and $B\in \F_{[1,\infty)}$. Then, we have that for all $n\geq 1$ and any $y_1^n\in A^n$
\begin{equation*}
P^x[\sigma^{-n}B \mid \eta_1^n=y_1^n]=P^{xy_1^n}[B].
\end{equation*}
\end{lemma}
\begin{proof}
By standard arguments of measure theory, it is enough to prove the result for a cylinder set $C=\{\eta_k^l=x_k^l\}$, for $1\leq k\leq l$.
Observe that $\sigma^{-n}C=\{\eta_{k+n}^{l+n}=x_k^l\}$. Therefore, we obtain
\begin{align*}
P^x[\sigma^{-n}C\mid \eta_1^n=y_1^n]&=P^x[\eta_{k+n}^{l+n}=x_k^l\mid \eta_1^n=y_1^n]\\
&=P^{xy_1^n}[\eta_{k}^{l}=x_k^l]\\
&=P^{xy_1^n}[C],
\end{align*}
which concludes the proof.
\end{proof}
In the next proposition, we denote by $\mathcal I$ the $\sigma$-algebra of invariant events under the shift $\sigma$.
\begin{lemma}
\label{lem4}
	Let $x\in \X_-$ such that $x_{-K+1}^0\notin \mathcal P_K$. Under assumption \eqref{ass1} of Theorem~\ref{theo1} we have for all $B\in \mathcal I$,
\begin{equation*}
P^x[B]=\sum_{n=1}^{\infty}\sum_{y_1^n\in \mathcal{E}_x^n}P^{xy_1^n}[B]P^x[\eta_1^n=y_1^n],
\end{equation*}
where $\mathcal{E}_x^n:=\{y_1^n\in A^n:(xy_1^n)_{-K+1}^0\in \mathcal P_K\; \text{and}\; (xy_1^j)_{-K+1}^0\notin \mathcal P_K\; \text{for}\; 1\leq j<n \}$.
\end{lemma}
\begin{remark}
It may happen that $\mathcal{E}_x^n=\emptyset$ for some $n\geq 1$. In this case, we use the convention that 
$\sum_{y_1^n\in \mathcal{E}_x^n}=0$.
\end{remark}
\begin{proof} We have by Definition \ref{Deferi} (iii), 
\begin{align*}
	P^x[B]&=\sum_{n=1}^{\infty} P^x[B,T_{\mathcal P_K}=n]\\
	&=\sum_{n=1}^{\infty}\sum_{y_1^n\in A^n} P^x[B,\eta_1^n=y_1^n,T_{\mathcal P_K}=n]\\
	&=\sum_{n=1}^{\infty}\sum_{y_1^n\in A^n} P^x[\sigma^{-n}B,\eta_1^n=y_1^n,T_{\mathcal P_K}=n]\\
	&=\sum_{n=1}^{\infty}\sum_{y_1^n\in \mathcal{E}_x^n} P^x[\sigma^{-n}B\mid \eta_1^n=y_1^n]P^x[\eta_1^n=y_1^n]\\
	&=\sum_{n=1}^{\infty}\sum_{y_1^n \in \mathcal{E}_x^n}P^{xy_1^n}[B]P^x[\eta_1^n=y_1^n],
\end{align*}
where in the third line we used the fact that $B=\sigma^{-n}B$ (since $B\in \mathcal I$), and Lemma~\ref{lem3} to obtain the last equality.
\end{proof}
\medskip

\noindent
{\it Proof of Theorem \ref{theo1}.}
\medskip

\noindent
We first prove that for any  $x,y\in \mathcal{X}^-$ such that $x_{-K+1}^0$ and $y_{-K+1}^0 \in \mathcal P_K$, we have  $P^x|_{\mathcal I} \sim P^y|_{\mathcal I}$. For this, it is enough to show that there exists $n\geq 1$ such that $P^x|_{\mathcal{F}_{[n,\infty)}} \sim P^y|_{\mathcal{F}_{[n,\infty)}}$. By symmetry, it will be enough to prove that $P^x|_{\mathcal{F}_{[n,\infty)}}\ll P^y|_{\mathcal{F}_{[n,\infty)}}$. 
To obtain the absolute continuity, we will use a result from the theory of predictable absolute continuity and singularity (ACS) criterion developed by Jacod and Shiryaev \citep{jacod/shiryaev/2002}.

 Fix $n\geq 1$. For $w \in \X$, let $Z_k(w) := \frac{dP^x|_{\F_{[n,n+k]}}}{dP^y|_{\F_{[n,n+k]}}}(w)$ and $\alpha_k(w) := Z_k(w)/Z_{k-1}(w)$, $k\geq 2$. Also, define for $k\geq 2$,
\begin{equation*}
 d_k(w) := E_{P^y}\big[(1-\sqrt{\alpha_k})^2 | \F_{[n, n+k-1]}\big](w). 
\end{equation*}
The predictable ACS criterion is given by
\begin{theo*}[see \cite{jacod/shiryaev/2002}, Theorem 2.36, p.253]
If for all $k \geq 1$ we have $P^x |_{\F_{[n,n+k]}} \ll P^y |_{\F_{[n,n+k]}}$, then $P^x|_{\F_{[n,\infty)}} \ll P^y|_{\F_{[n,\infty)}}$ if and only if $\sum_{k=2}^\infty d_k(w) < \infty$, $P^x$-a.s.
\end{theo*}
By Proposition \ref{proplac1}, there exists $n\geq 1$ such that for all $k\geq1$, $P^x|_{\F_{[n,n+k]}} \ll P^y|_{\F_{[n,n+k]}}$. Now, let us rewrite $d_k$ in a more explicit form. We have that
\begin{equation*}
 \frac{dP^{x}|_{\F_{[n,n+k]}}}{dP^x|_{\F_{[n,n+k-1]}}}(w) =P^x[\eta_n =w_{n+k} |\eta^{n+k-1}_1=w^{n+k-1}_1]=: P^x(w_{n+k} | w^{n+k-1}_1).
\end{equation*}

Hence, for all $k\geq 2$
\begin{align*}
d_k(w) &= \sum_{w_{n+k} \in A} \left(1-\sqrt{\frac{P^x(w_{n+k} | w^{n+k-1}_1)}{P^y(w_{n+k} | w^{n+k-1}_1)}}\right)^2 P^y(w_{n+k} | w^{n+k-1}_1)\\
 &= \sum_{w_{n+k} \in A} \left(\sqrt{P^x(w_{n+k} | w^{n+k-1}_1)}-\sqrt{P^y(w_{n+k} |w^{n+k-1}_1)}\right) ^2\\
&=\sum_{w_{n+k} \in A} \left(\sqrt{ g(x w^{n+k}_1)}-\sqrt{g(yw^{n+k}_1)}\right) ^2
\end{align*}
where in the first line we use the convention $\frac{0}{0}=0$. Since $\sqrt{g}\in \mathcal{W}^{2}$, we obtain that $\sum_{k\geq 2}d_k(w)<\infty$ for $P^x$-a.e.~$w\in \X$.
Thus, we can apply the ACS criterion and obtain that $P^x|_{\mathcal{F}_{[n,\infty)}} \ll P^y|_{\mathcal{F}_{[n,\infty)}}$. By symmetry, we immediately obtain $P^x|_{\mathcal{F}_{[n,\infty)}} \sim P^y|_{\mathcal{F}_{[n,\infty)}}$ and therefore $P^x|_{\mathcal I} \sim P^y|_{\mathcal I}$ for any  $x,y\in \mathcal{X}^-$ such that $x_{-K+1}^0$ and $y_{-K+1}^0 \in \mathcal P_K$.

Now, consider $z\in \X_-$ such that $z_{-K+1}^0\notin \mathcal P_K$. We will show that $P^z|_{\mathcal I} \sim P^x|_{\mathcal I}$ for $x\in \mathcal{X}^-$ such that $x_{-K+1}^0\in \mathcal P_K$. Using Lemma \ref{lem4}, we have for all $B\in \mathcal I$
\[
P^z[B]=\sum_{n=1}^{\infty}\sum_{y_1^n\in \mathcal{E}_x^n}P^{zy_1^n}[B]P^z[\eta_1^n=y_1^n].
\]
From this last expression, we deduce that if $P^z[B]>0$ then there exists $n\geq 1$ and $y_1^n\in \mathcal{E}_x^n$ such that $P^{zy_1^n}[B]>0$. Since $(zy_1^n)_{-K+1}^0\in \mathcal P_K$, we have that $P^{zy_1^n}|_{\mathcal{I}}\sim P^x|_{\mathcal I}$ and thus $P^x[B]>0$. Similarly, if $P^x[B]>0$, for all $n\geq 1$ and $y_1^n\in  \mathcal{E}_x^n$ we have $P^{zy_1^n}[B]>0$ and thus $P^z[B]>0$. 

The first part of the proof shows that for all $x,y\in \X_-$, $P^x|_{\mathcal{I}}\sim P^y|_{\mathcal{I}}$. Suppose that there exists two stationary probability measures $P^\mu$ and $P^\nu$ compatible with~$g$. By the Ergodic Decomposition Theorem (cf.~Appendix),  we can assume that  $P^\mu$ and $P^\nu$  are ergodic. Then, as $ P^\mu\neq P^\nu$, there exists an invariant set $I$ such that $P^\mu[I]=1-P^\nu[I]=1$.
 Since $P^\mu[I]=\int P^x[I]\mu(dx)$, we have that $P^x[I]=1$, $\mu(dx)$-a.s. In the same way, we have that $P^y[I]=0$, $\nu(dy)$-a.s. Therefore, we can pick up $x$ such that $P^x[I]=1$ and $y$ such that $P^y[I]=0$. But this contradict the fact that $P^x|_{\mathcal I} \sim P^y|_{\mathcal I}$ for all $x, y \in \X_-$. 
Therefore there exists at most one stationary measure $\mu$ compatible with $g$. 

Finally, let us show that $\mu$ is weak Bernoulli when it exists. 
To this end, we use Corollary 2.8 in \cite{tong/vanrandel/2014}. In our notation, it states that $\mu$ is weak Bernoulli if and only if for $\mu|_{\F^-} \otimes \mu|_{\F^-}$-a.e.~$(x,y)$, there exist a $k_0 \geq 0$ such that $P^{x}$ and $P^{y}$ are not mutually singular on $\F_{[k_0,\infty)}$. By Proposition \ref{proplac1},  this is the case for all  $x, y\in \X_-$ such that $x_{-K+1}^0$ and $y_{-K+1}^0 \in \mathcal P_K$. Therefore, it remains to show that $\mu[\eta_{-K+1}^{0}=x_{-K+1}^0]=0$  for $x_{-K+1}^0\notin \mathcal{P}_K$. Observe that for all $n\geq 1$,
\begin{align*}
\mu[\eta_{-K+1}^0=x_{-K+1}^0]&=\mu[\eta_{n-K+1}^n=x_{-K+1}^0]\\
&=\mu[\eta_{n-K+1}^n=x_{-K+1}^0, T_{{\mathcal P}_K}>n]\\
&\leq \mu[T_{{\mathcal P}_K}>n].
\end{align*}
By assumption \eqref{ass1}, this last term tends to 0 as $n\to \infty$.
\qed

\begin{proof}[Proof of Corollary \ref{coro2}]
We only have to show that $d_k$ from the ACS criterion can be bounded by 
$$\sum_{w_{n+k} \in A} \left(g(x w^{n+k}_1)-g(yw^{n+k}_1)\right) ^2$$
instead of 
$$\sum_{w_{n+k} \in A} \left(\sqrt{ g(x w^{n+k}_1)}-\sqrt{g(yw^{n+k}_1)}\right) ^2.$$

Let us first define $\gamma:=\inf_x\{g(x):g(x)>0\}$ and observe that $\gamma>0$ since $g$ is continuous and $\{g>0\}$ is compact (this last fact follows from the equality $\{g>0\}=(\mathcal{D}^{(<\infty)})^c$ and the property that $\mathcal{D}^{(<\infty)}$ is open in the product topology). Now, consider the $K$ used in the assumptions of Theorem \ref{theo1}. We have, for $k\geq K$,
\begin{align*}
d_k(w)&=\sum_{w_{n+k} \in A}\left(\sqrt{g(xw^{n+k}_1)}-\sqrt{g(yw^{n+k}_1)}\right) ^2\nonumber\\
&\leq \frac{1}{4\gamma}\sum_{w_{n+k} \in A}\left((g(xw^{n+k}_1)-g(yw^{n+k}_1)\right)^2.
\end{align*}
 By hypothesis $g\in \mathcal{W}^2$, therefore we have that $\sum_{k\geq 2}d_k(w)<\infty$, for $P^x$-a.e.~$w$. Finally, when $A$ is finite, the continuity of $g$ guarantees the existence of a stationary probability measure compatible with $g$, see for example \cite{keane/1972}.
\end{proof}

\section{Appendix}
In this section, for completeness, we state and prove the Ergodic Decomposition Theorem for invariant probability measures compatible with some kernel $g$. This result is standard in statistical mechanics (see for example \cite{georgii2011gibbs}), but we could not find a published proof of this result in the context of $g$-measures. 

We first recall some definitions from \cite{fernandez/maillard/2005}. 
We denote by $\mathcal{Z}_b$ the set of finite intervals contained in $\Z$. For $\Lambda \in \mathcal{Z}_b$, let $l_\Lamb:=\min\Lambda$ and $m_{\Lamb}:=\max\Lamb$.
\begin{defi}
A left interval-specification (L.I.S.) on $(\X, \F)$ is a family of probability kernels $f=(f_{\Lamb})_{\Lamb\in \mathcal{Z}_b }$, $f_{\Lamb}:\F_{(-\infty,\;m_{\Lamb}]}\times \X\to [0,1]$ such that for all $\Lambda\in \mathcal{Z}_b$, 
\begin{itemize}
\item[(i)] For $B\in \F_{(-\infty , m_{\Lamb}]}$, $f_\Lamb(B\mid \cdot)$ is $\F_{(-\infty,\; l_{\Lamb})}$-measurable; 
\item[(ii)] For $C\in \F_{(-\infty , \; l_{\Lamb})}$, $f_\Lamb(C\mid \cdot)={\bf 1}_C(\cdot)$;
\item[(iii)] For $\Delta \in \mathcal{Z}_b$, $\Delta \supset \Lamb$, $f_{\Delta}f_{\Lamb}=f_\Delta$, on $\F_{(-\infty ,\; m_{\Lamb}]}$. (Here $f_{\Delta}f_{\Lamb}$ is the usual composition of kernels)

\end{itemize}
\end{defi}

\begin{defi}
A probability measure on $(\X, \F)$ is compatible (or consistent) with the L.I.S. $f=(f_{\Lamb})_{\Lamb\in \mathcal{Z}_b }$ if for each $\Lambda\in \mathcal{Z}_b $,
\[
\mu f_\Lamb=\mu, \;\;\; {\text{on}}\;\; \F_{(-\infty ,\; m_\Lamb]}.
\]
\end{defi}
A L.I.S.~$f=(f_{\Lamb})_{\Lamb\in \mathcal{Z}_b }$  is {\it shift-invariant} if
$$
f_{\Lambda-1}(\sigma  B\mid \sigma  w)=f_{\Lambda}(B\mid w)
$$
for all $B\in \F_{(-\infty,m_\Lamb]}$, $\Lamb\in \mathcal{Z}_b$ and $w\in \X$.


Let us also define
\[
\mathcal{G}(f)=\{\mu: \mu f_{\Lambda}(B)=\mu(B), \forall B\in \F_{(-\infty,\;m_{\Lambda}]}, \forall \Lambda \in \mathcal{Z}_b \},
\] 
the set of probability measures compatible with $f$ and 
\[
\mathcal{G}_\theta(f):=\mathcal{P}_{\theta}(\X, \F) \cap \mathcal{G}(f)
\]
where  $\mathcal{P}_{\theta}(\X, \F)$ is the set of stationary probability measures on $(\X, \F)$.
\medskip

\noindent
Denoting by $\text{ex}(\mathcal{G}_{\theta}(f))$ the set of extremal points of $\mathcal{G}_{\theta}(f)$ we have the following
\begin{theo*}[Ergodic Decomposition Theorem]
Suppose that $(A, \B(A))$ is standard Borel and $f$ is shift-invariant. If $\mathcal{G}_\theta(f)\neq \emptyset$, then $\emph{ex}(\mathcal{G}_{\theta}(f))\neq \emptyset$ and, for each $\mu\in \mathcal{G}_\theta(f)$, there exists a unique probability measure $w_{\mu}$ on $\emph{ex}(\mathcal{G}_{\theta}(f))$ such that
\begin{equation}
\label{Choquetdecomp}
\mu=\int_{\emph{ex}(\mathcal{G}_\theta(f))}\nu \;w_{\mu}(d\nu).
\end{equation}
Moreover, $\emph{ex}(\mathcal{G}_{\theta}(f))$ coincides with the set of ergodic probability measures compatible with $f$.
\end{theo*}
\medskip

\noindent
{\it Proof.} The proof follows the lines of the proof of Theorem 14.17 
in \cite{georgii2011gibbs}. We first 
recall the definition of the key object $\pi$ from that proof. Following Definition (7.21) in \cite{georgii2011gibbs}, a probability kernel $\pi$ from $(\X, \mathcal{I})$ to $(\X, \F)$ is a $(\mathcal{G}_\theta(f),\mathcal{I})$-{\it kernel} if it satisfies
\begin{itemize}
    \item[(i)] for all $\mu\in \mathcal{G}_\theta(f)$ and $B\in \mathcal{F}$, $\pi^{\cdot}(B) = \mu(B\mid\mathcal{I})(\cdot)$, $\mu$-a.s.;
    \item[(ii)] $\{x\in \X: \pi^x(\cdot)\in \mathcal{G}_\theta(f)\}\in \mathcal{I}$;
    \item[(iii)] for all $\mu\in \mathcal{G}_\theta(f)$, $\mu(\{x\in \X: \pi^x(\cdot)\in \mathcal{G}_\theta(f)\})=1$.
\end{itemize}
By Proposition 7.22 in \cite{georgii2011gibbs} the decomposition (\ref{Choquetdecomp}) then follows from the existence of such a kernel once we observe that the elements of $\text{ex}(\mathcal{G}_{\theta}(f))$ are the ergodic probability measures compatible with $f$ (see Theorem 3.5 item (a) of \cite{fernandez/maillard/2005}). \\
As $(A,\mathcal{B}(A))$ is standard Borel, Theorem 14.10 in \cite{georgii2011gibbs} establishes the existence of a $(\mathcal{P}_{\theta}(\X, \F),\mathcal{I})$-kernel $\pi$ (the definition of a $(\mathcal{P}_{\theta}(\X, \F),\mathcal{I})$-kernel is analogous to that of a $(\mathcal{G}_\theta(f),\mathcal{I})$-kernel given above, one simply replaces $\mathcal{G}_\theta(f)$  by $\mathcal{P}_{\theta}(\X, \F)$ in the definition). Therefore, as $\mathcal{G}_\theta(f)=\mathcal{P}_{\theta}(\X, \F) \cap \mathcal{G}(f)$, it only remains to show that $\pi^\cdot \in \mathcal{G}(f)$, $\mu$-a.s., i.e.\ that the conditional distributions given $\mathcal{I}$ are themselves compatible with $f$. This last property does not actually depend on the particular construction of $\pi$ but only on the fact that $\pi$ is a $(\mathcal{P}_{\theta}(\X, \F),\mathcal{I})$-kernel. Indeed, consider $\mu\in\mathcal{G}_\theta(f)$.
Since $(A,\mathcal{B}(A))$ is standard Borel, there exists a countable class of sets, let us say $\mathcal{A}$, stable under finite intersections, that generates $\mathcal{B}(A)$.  Now, consider the class $\mathcal{C}$  of (finite) cylinder sets in $\F$ with basis in $\mathcal{A}$. Observe that $\mathcal{C}$ is stable under finite intersections and generates $\F$.
Now, fix $\Lamb\in \mathcal{Z}_b$ and $B\in \F_{(-\infty,\; m_\Lamb]} \cap \mathcal{C}$. Denoting $\mathcal{T}_{-\infty}:=\cap_{n\geq 1}\F_{(-\infty, -n]}$ and using Lemma 6.5 of \cite{fernandez/maillard/2005} we obtain that
\begin{align*}
\pi^\cdot\gamma_\Lamb(B)&=\mu( \gamma_\Lamb(B\mid \cdot \;) \mid \mathcal{I})\\
&=\mu( \mu(B \mid \F_{(-\infty,\; l_\Lamb)}) \mid \mathcal{I})\\
&=\mu( \mu(B\mid \mathcal{T}_{-\infty}) \mid \mathcal{I})\\
&=\mu(B\mid \mathcal{I})=\pi^\cdot (B), \;\;\; \mu \text{-a.s.}
\end{align*}
Since $ \mathcal{Z}_b$ and $\mathcal{C}$ are countable, by the monotone class lemma,  we deduce that $\mu$-a.s.,
for all $\Lambda\in  \mathcal{Z}_b$ and $C \in \F_{(-\infty,\; m_\Lamb]}$, $\pi^\cdot\gamma_\Lamb(C)=\pi^\cdot (C)$, which shows that $\mu(\pi^\cdot \in  \mathcal{G}(f))=1$.
\qed

\section*{Acknowledgments}
C.G.~thanks FAPESP (grant 2023/07228-9)  and FAEPEX (grant 3073/23) for financial supports. D.Y.T.~thanks CNPq (grant 421955/2023-6). S.G.~thanks FAPESP (fellowship 24/06341-9 and grants 23/13453-5 and 23/07507-5) and CNPq (fellowship 314909/2023-0 and grant 441884/2023-7). D.Y.T and S.G. 
~thanks FAPESP  Research, Innovation and Dissemination 
Center for Neuromathematics (grant 2013/ 07699-0).
\bibliographystyle{alpha}
\bibliography{IPbibli}

\end{document}